\documentclass[12pt,a4paper]{article}

\title{Asymptotics for Quasi-Stationary Distributions of Perturbed Discrete Time Semi-Markov Processes}
\author{Mikael Petersson\footnote{Department of Mathematics, Stockholm University, SE-106 91 Stockholm, Sweden, mikpe@math.su.se.}}
\date{}

\usepackage{amsmath}
\usepackage{amssymb}
\usepackage{amsthm}

\newtheorem{theorem}{Theorem}
\newtheorem{lemma}{Lemma}

\begin{document}

\maketitle

\begin{abstract}
In this paper, we study quasi-stationary distributions of nonlinearly perturbed semi-Markov processes in discrete time. This type of distributions is of interest for the analysis of stochastic systems which have finite lifetimes, but are expected to persist for a long time. We obtain asymptotic power series expansions for quasi-stationary distributions and it is shown how the coefficients in these expansions can be computed from a recursive algorithm. As an illustration of this algorithm, we present a numerical example for a discrete time Markov chain.
\end{abstract}


\section{Introduction}

This paper is a sequel of Petersson (2016) where recursive algorithms for computing asymptotic expansions of moment functionals for non-linearly perturbed semi-Markov processes in discrete time where presented. Here, these expansions play a fundamental role for constructing asymptotic expansions of quasi-stationary distributions for such processes. Let us remark that all notation, conditions, and key results which we need here are repeated. However, some extensive formulas needed for computation of coefficients in certain asymptotic expansions are not repeated. Thus, the present paper is essentially self-contained.

Quasi-stationary distributions are useful for studies of stochastic systems with random lifetimes. Usually, for such systems, the evolution of some quantity of interest is described by some stochastic process and the lifetime of the system is the first time this process hits some absorbing subset of the state space. For such processes, the stationary distribution will be concentrated on this absorbing subset. However, if we expect that the system will persist for a long time, the stationary distribution may not be an appropriate measure for describing the long time behaviour of the process. Instead, it might be more relevant to consider so-called quasi-stationary distributions. This type of distributions is obtained by taking limits of transition probabilities which are conditioned on the event that the process has not yet been absorbed.

Models of the type described above arise in many areas of applications such as epidemics, genetics, population dynamics, queuing theory, reliability, and risk theory. For example, in population dynamics models the number of individuals may be modelled by some stochastic process and we can consider the extinction time of the population as the lifetime. In epidemic models, the process may describe the evolution of the number of infected individuals and we can regard the end of the epidemic as the lifetime.

We consider, for every $\varepsilon \geq 0$, a discrete time semi-Markov process $\xi^{(\varepsilon)}(n)$, $n=0,1,\ldots,$ on a finite state space $X = \{ 0,1,\ldots,N \}$. It is assumed that the process $\xi^{(\varepsilon)}(n)$ depends on $\varepsilon$ in such a way that its transition probabilities are functions of $\varepsilon$ which converge pointwise to the transition probabilities for the limiting process $\xi^{(0)}(n)$. Thus, we can interpret $\xi^{(\varepsilon)}(n)$, for $\varepsilon > 0$, as a perturbation of $\xi^{(0)}(n)$. Furthermore, it is assumed that the states $\{ 1,\ldots,N \}$ is a communicating class for $\varepsilon$ small enough.

Under conditions mentioned above, some additional assumptions of finite exponential moments of distributions of transition times, and a condition which guarantees that the limiting semi-Markov process is non-periodic, a unique quasi-stationary distribution, independent of the initial state, can be defined for each sufficiently small $\varepsilon$ by the following relation,
\begin{equation*}
 \pi_j^{(\varepsilon)} = \lim_{n \rightarrow \infty} \mathsf{P}_i \{ \xi^{(\varepsilon)}(n) = j \, | \, \mu_0^{(\varepsilon)} > n \}, \ i,j \neq 0,
\end{equation*}
where $\mu_0^{(\varepsilon)}$ is the first hitting time of state $0$.

In the present paper, we are interested in the asymptotic behaviour of the quasi-stationary distribution as the perturbation parameter $\varepsilon$ tends to zero. Specifically, an asymptotic power series expansion of the quasi-stationary distribution is constructed.

We allow for nonlinear perturbations, i.e., the transition probabilities may be nonlinear functions of $\varepsilon$. We do, however, restrict our consideration to smooth perturbations by assuming that certain mixed power-exponential moment functionals for transition probabilities, up to some order $k$, can be expanded in asymptotic power series with respect to $\varepsilon$.

In this case, we show that the quasi-stationary distribution has the following asymptotic expansion,
\begin{equation} \label{eq:expansionintro}
 \pi_j^{(\varepsilon)} = \pi_j^{(0)} + \pi_j[1] \varepsilon + \cdots + \pi_j[k] \varepsilon^k + o(\varepsilon^k), \ j \neq 0,
\end{equation}
where the coefficients $\pi_j[1],\ldots,\pi_j[k]$, can be calculated from explicit recursive formulas. These formulas are functions of the coefficients in the expansions of the moment functionals mentioned above. The existence of the expansion \eqref{eq:expansionintro} and the algorithm for computing the coefficients in this expansion is the main result of this paper.

It is worth mentioning that the asymptotic relation given by equation \eqref{eq:expansionintro} simultaneously cover three different cases. In the simplest case, there exists $\varepsilon_0 > 0$ such that transitions to state $0$ are not possible for any $\varepsilon \in [0,\varepsilon_0]$. In this case, relation \eqref{eq:expansionintro} gives asymptotic expansions for stationary distributions. Then, we have an intermediate case where transitions to state $0$ are possible for all $\varepsilon \in (0,\varepsilon_0]$ but not possible for $\varepsilon = 0$. In this case we have that $\mu_0^{(\varepsilon)} \rightarrow \infty$ in probability as $\varepsilon \rightarrow 0$. In the mathematically most difficult case, we have that transitions to state $0$ are possible for all $\varepsilon \in [0,\varepsilon_0]$. In this case, the random variables $\mu_0^{(\varepsilon)}$ are stochastically bounded as $\varepsilon \rightarrow 0$.

The expansion \eqref{eq:expansionintro} is given for continuous time semi-Markov processes in Gyllenberg and Silvestrov (1999, 2008). However, the discrete time case is interesting in its own right and deserves a special treatment. In particular, a discrete time model is often a natural choice in applications where measures of some quantity of interest are only available at given time points, for example days or months. The proof of the result for the continuous time case, as well as the proofs in the present paper, is based on the theory of non-linearly perturbed renewal equations. For results related to continuous time in this line of research, we refer to the comprehensive book by Gyllenberg and Silvestrov (2008), which also contains an extensive bibliography of work in related areas. The corresponding theory for discrete time renewal equations has been developed in Gyllenberg and Silvestrov (1994), Englund and Silvestrov (1997), Silvestrov and Petersson (2013), and Petersson (2014a, b, 2015).

Quasi-stationary distributions have been studied extensively since the 1960's. For some of the early works on Markov chains and semi-Markov processes, see, for example, Vere-Jones (1962), Kingman (1963), Darroch and Seneta (1965), Seneta and Vere-Jones (1966), Cheong (1968, 1970), and Flaspohler and Holmes (1972). A survey of quasi-stationary distributions for models with discrete state spaces and more references can be found in van Doorn and Pollett (2013).

Studies of asymptotic properties for first hitting times, stationary distributions, and other characteristics for Markov chains with linear, polynomial, and analytic perturbations have attracted a lot of attention, see, for example, Simon and Ando (1961), Schweitzer (1968), Ga{\u \i}tsgori and Pervozvanski{\u \i} (1975), Courtois and Louchard (1976), Latouche and Louchard (1978), Delebecque (1983), Latouche (1991), Stewart (1991), Hassin and Haviv (1992), Yin and Zhang (1998, 2003), Altman, Avrachenkov, and N{\'u}{\~n}ez-Queija (2004), Avrachenkov and Haviv (2004), and Avrachenkov, Filar, and Howlett (2013). Recently, some of the results of these papers have been extended to non-linearly perturbed semi-Markov processes. Using a method of sequential phase space reduction, asymptotic expansions for expected first hitting times and stationary distributions are given in Silvestrov and Silvestrov (2015). This paper also contains an extensive bibliography.

Let us now briefly comment on the structure of the present paper. In Section \ref{sec:mainresult}, most of the notation we need are introduced and the main result is formulated. We apply the discrete time renewal theorem in order to get a formula for the quasi-stationary distribution in Section \ref{sec:quasistationary} and then the proof of the main result is presented in Section \ref{sec:proof}. Finally, in Section \ref{sec:markovchains}, we illustrate the results in the special case of discrete time Markov chains.

\section{Main Result} \label{sec:mainresult}

In this section we first introduce most of the notation that will be used in the present paper and then we formulate the main result. 

For each $\varepsilon > 0$, let $\xi^{(\varepsilon)}(n)$, $n=0,1,\ldots,$ be a discrete time semi-Markov process on the state space $X = \{0,1,\ldots,N\}$, generated by the discrete time Markov renewal process $(\eta_n^{(\varepsilon)},\kappa_n^{(\varepsilon)})$, $n=0,1,\ldots,$ having state space $X \times \{ 1,2,\ldots \}$ and transition probabilities
\begin{equation*}
 Q_{i j}^{(\varepsilon)}(k) = \mathsf{P} \{ \eta_{n+1}^{(\varepsilon)} = j, \ \kappa_{n+1}^{(\varepsilon)} = k \, | \, \eta_n^{(\varepsilon)} = i, \ \kappa_n^{(\varepsilon)} = l \}, \ k,l = 1,2,\ldots, \ i,j \in X.
\end{equation*}
We can write the transition probabilities as $Q_{i j}^{(\varepsilon)}(k) = p_{i j}^{(\varepsilon)} f_{i j}^{(\varepsilon)}(k)$, where $p_{i j}^{(\varepsilon)}$ are transition probabilities for the embedded Markov chain $\eta_n^{(\varepsilon)}$ and
\begin{equation*}
 f_{i j}^{(\varepsilon)}(k) = \mathsf{P} \{ \kappa_{n+1}^{(\varepsilon)} = k \, | \, \eta_n^{(\varepsilon)} = i, \ \eta_{n+1}^{(\varepsilon)} = j \}, \ k = 1,2,\ldots, \ i,j \in X,
\end{equation*}
are conditional distributions of transition times.

Let us here remark that definitions of discrete time semi-Markov processes and Markov renewal processes can be found in, for example, Petersson (2016).

For each $j \in X$, let $\nu_j^{(\varepsilon)} = \min \{ n \geq 1 : \eta_n^{(\varepsilon)} = j \}$ and $\mu_j^{(\varepsilon)} = \kappa_1^{(\varepsilon)} + \cdots + \kappa^{(\varepsilon)}_{\nu_j^{(\varepsilon)}}$. By definition, $\nu_j^{(\varepsilon)}$ and $\mu_j^{(\varepsilon)}$ are the first hitting times of state $j$ for the embedded Markov chain and the semi-Markov process, respectively.

In what follows, we use $\mathsf{P}_i$ and $\mathsf{E}_i$ to denote probabilities and expectations conditioned on the event $\{ \eta_0^{(\varepsilon)} = i \}$.

Let us define
\begin{equation*}
 g_{i j}^{(\varepsilon)}(n) = \mathsf{P}_i \{ \mu_j^{(\varepsilon)} = n, \ \nu_0^{(\varepsilon)} > \nu_j^{(\varepsilon)} \}, \ n =0,1,\ldots, \ i,j \in X,
\end{equation*}
and
\begin{equation*}
 g_{i j}^{(\varepsilon)} = \mathsf{P}_i \{ \nu_0^{(\varepsilon)} > \nu_j^{(\varepsilon)} \}, \ i,j \in X.
\end{equation*}
The functions $g_{i j}^{(\varepsilon)}(n)$ define discrete probability distributions which may be improper, i.e., $\sum_{n=0}^\infty g_{i j}^{(\varepsilon)}(n) = g_{i j}^{(\varepsilon)} \leq 1$.

Let us also define the following mixed power-exponential moment functionals,
\begin{equation*}
 p_{i j}^{(\varepsilon)}(\rho,r) = \sum_{n=0}^\infty n^r e^{\rho n} Q_{i j}^{(\varepsilon)}(n), \ \rho \in \mathbb{R}, \ r=0,1,\ldots, \ i,j \in X,
\end{equation*}
\begin{equation*}
 \phi_{i j}^{(\varepsilon)}(\rho,r) = \sum_{n=0}^\infty n^r e^{\rho n} g_{i j}^{(\varepsilon)}(n), \ \rho \in \mathbb{R}, \ r=0,1,\ldots, \ i,j \in X,
\end{equation*}
\begin{equation*}
 \omega_{i j s}^{(\varepsilon)}(\rho,r) = \sum_{n=0}^\infty n^r e^{\rho n} h_{i j s}^{(\varepsilon)}(n), \ \rho \in \mathbb{R}, \ r=0,1,\ldots, \ i,j,s \in X,
\end{equation*}
where $h_{i j s}^{(\varepsilon)}(n) = \mathsf{P}_i \{ \xi^{(\varepsilon)}(n) = s, \ \mu_0^{(\varepsilon)} \wedge \mu_j^{(\varepsilon)} > n \}$. For convenience, we denote
\begin{equation*}
 p_{i j}^{(\varepsilon)}(\rho) = p_{i j}^{(\varepsilon)}(\rho,0), \quad \phi_{i j}^{(\varepsilon)}(\rho) = \phi_{i j}^{(\varepsilon)}(\rho,0), \quad \omega_{i j s}^{(\varepsilon)}(\rho) = \omega_{i j s}^{(\varepsilon)}(\rho,0).
\end{equation*}

We now introduce the following conditions:

\begin{enumerate}
\item[$\mathbf{A}$:]
\begin{enumerate}
\item[$\mathbf{(a)}$] $p_{i j}^{(\varepsilon)} \rightarrow p_{i j}^{(0)}$, as $\varepsilon \rightarrow 0$, $i \neq 0$, $j \in X$.
\item[$\mathbf{(b)}$] $f_{i j}^{(\varepsilon)}(n) \rightarrow f_{i j}^{(0)}(n)$, as $\varepsilon \rightarrow 0$, $n=1,2,\ldots,$ $i \neq 0$, $j \in X$.
\end{enumerate}
\end{enumerate}

\begin{enumerate}
\item[$\mathbf{B}$:] $g_{i j}^{(0)} > 0$, $i,j \neq 0$.
\end{enumerate}

\begin{enumerate}
\item[$\mathbf{C}$:] There exists $\beta > 0$ such that:
\begin{enumerate}
\item[$\mathbf{(a)}$] $\limsup_{0 \leq \varepsilon \rightarrow 0} p_{i j}^{(\varepsilon)}(\beta) < \infty$, for all $i \neq 0$, $j \in X$.
\item[$\mathbf{(b)}$] $\phi_{i i}^{(0)}(\beta_i) \in (1,\infty)$, for some $i \neq 0$ and $\beta_i \leq \beta$.
\end{enumerate}
\end{enumerate}

\begin{enumerate}
\item[$\mathbf{D}$:] $g_{i i}^{(0)}(n)$ is a non-periodic distribution for some $i \neq 0$.
\end{enumerate}

Under the conditions stated above, there exists, for sufficiently small $\varepsilon$, so-called quasi-stationary distributions, which are independent of the initial state $i \neq 0$, and given by the relation
\begin{equation} \label{eq:qsddef}
 \pi_j^{(\varepsilon)} = \lim_{n \rightarrow \infty} \mathsf{P}_i \{ \xi^{(\varepsilon)}(n) = j \, | \, \mu_0^{(\varepsilon)} > n \}, \ j \neq 0.
\end{equation}

An important role for the quasi-stationary distribution is played by the following characteristic equation,
\begin{equation} \label{eq:chareq}
 \phi_{i i}^{(\varepsilon)}(\rho) = 1,
\end{equation}
where $i \neq 0$ is arbitrary.

The following lemma summarizes some important properties for the root of equation \eqref{eq:chareq}. A proof is given in Petersson (2016).

\begin{lemma} \label{lmm:chareq}
 Under conditions $\mathbf{A}$--$\mathbf{C}$ there exists, for sufficiently small $\varepsilon$, a unique non-negative solution $\rho^{(\varepsilon)}$ of the characteristic equation \eqref{eq:chareq} which is independent of $i$. Moreover, $\rho^{(\varepsilon)} \rightarrow \rho^{(0)}$, as $\varepsilon \rightarrow 0$.
\end{lemma}

In order to construct an asymptotic expansion for the quasi-stationary distribution, we need a perturbation condition for the transition probabilities $Q_{i j}^{(\varepsilon)}(k)$ which is stronger than $\mathbf{A}$. This condition is formulated in terms of the moment functionals $p_{i j}^{(\varepsilon)}(\rho^{(0)},r)$.

\begin{enumerate}
\item[$\mathbf{P_k}$:] $p_{i j}^{(\varepsilon)}(\rho^{(0)},r) = p_{i j}^{(0)}(\rho^{(0)},r) + p_{i j}[\rho^{(0)},r,1] \varepsilon + \cdots + p_{i j}[\rho^{(0)},r,k-r] \varepsilon^{k-r} + o(\varepsilon^{k-r})$, for $r=0,\ldots,k$, $i \neq 0$, $j \in X$, where $|p_{i j}[\rho^{(0)},r,n]| < \infty$, for $r=0,\ldots,k$, $n=1,\ldots,k-r$, $i \neq 0$, $j \in X$.
\end{enumerate}

The following theorem is the main result of this paper. The proof is given in Section \ref{sec:proof}.

\begin{theorem} \label{thm:mainresult}
If conditions $\mathbf{A}$--$\mathbf{D}$ and $\mathbf{P_{k+1}}$ hold, then we have the following asymptotic expansion,
\begin{equation*}
 \pi_j^{(\varepsilon)} = \pi_j^{(0)} + \pi_j[1] \varepsilon + \cdots + \pi_j[k] \varepsilon^k + o(\varepsilon^k), \ j \neq 0,
\end{equation*}
where $\pi_j[n]$, $n=1,\ldots,k$, $j \neq 0$, can be calculated from a recursive algorithm which is described in Section \ref{sec:proof}.
\end{theorem}

\section{Quasi-Stationary Distributions} \label{sec:quasistationary}

In this section we use renewal theory in order to get a formula for the quasi-stationary distribution.

The probabilities $P_{i j}^{(\varepsilon)}(n) = \mathsf{P}_i \{ \xi^{(\varepsilon)}(n) = j, \ \mu_0^{(\varepsilon)} > n \}$, $i,j \neq 0$, satisfy the following discrete time renewal equation,
\begin{equation} \label{eq:renewaleq}
 P_{i j}^{(\varepsilon)}(n) = h_{i j}^{(\varepsilon)}(n) + \sum_{k=0}^n P_{i j}^{(\varepsilon)}(n-k) g_{i i}^{(\varepsilon)}(k), \ n=0,1,\ldots,
\end{equation}
where
\begin{equation*}
 h_{i j}^{(\varepsilon)}(n) = \mathsf{P}_i \{ \xi^{(\varepsilon)}(n) = j, \ \mu_0^{(\varepsilon)} \wedge \mu_i^{(\varepsilon)} > n \}.
\end{equation*}
Since $\sum_{n=0}^\infty g_{i i}^{(\varepsilon)}(n) = g_{i i}^{(\varepsilon)} \leq 1$, relation \eqref{eq:renewaleq} defines a possibly improper renewal equation.

Let us now, for each $n=0,1,\ldots,$ multiply both sides of \eqref{eq:renewaleq} by $e^{\rho^{(\varepsilon)} n}$, where $\rho^{(\varepsilon)}$ is the root of the characteristic equation $\phi_{i i}^{(\varepsilon)}(\rho) = 1$. Then, we get
\begin{equation} \label{eq:properrenewaleq}
 \widetilde{P}_{i j}^{(\varepsilon)}(n) = \widetilde{h}_{i j}^{(\varepsilon)}(n) + \sum_{k=0}^n \widetilde{P}_{i j}^{(\varepsilon)}(n-k) \widetilde{g}_{i i}^{(\varepsilon)}(k), \ n=0,1,\ldots,
\end{equation}
where
\begin{equation*}
 \widetilde{P}_{i j}^{(\varepsilon)}(n) = e^{\rho^{(\varepsilon)} n} P_{i j}^{(\varepsilon)}(n), \ \widetilde{h}_{i j}^{(\varepsilon)}(n) = e^{\rho^{(\varepsilon)} n} h_{i j}^{(\varepsilon)}(n), \ \widetilde{g}_{i i}^{(\varepsilon)}(n) = e^{\rho^{(\varepsilon)} n} g_{i i}^{(\varepsilon)}(n).
\end{equation*}

By the definition of the root of the characteristic equation, relation \eqref{eq:properrenewaleq} defines a proper renewal equation.

In order to prove our next result, we first formulate an auxiliary lemma. A proof can be found in Petersson (2016).

\begin{lemma} \label{lmm:convergence}
Assume that conditions $\mathbf{A}$--$\mathbf{C}$ hold. Then there exists $\delta > \rho^{(0)}$ such that:
\begin{enumerate}
\item[$\mathbf{(i)}$] $\phi_{k j}^{(\varepsilon)}(\rho) \rightarrow \phi_{k j}^{(0)}(\rho) < \infty$, as $\varepsilon \rightarrow 0$, $\rho \leq \delta$, $k,j \neq 0$.
\item[$\mathbf{(ii)}$] $\omega_{k j s}^{(\varepsilon)}(\rho) \rightarrow \omega_{k j s}^{(0)}(\rho) < \infty$, as $\varepsilon \rightarrow 0$, $\rho \leq \delta$, $k,j,s \neq 0$.
\end{enumerate}
\end{lemma}

We can now use the classical discrete time renewal theorem in order to get a formula for the quasi-stationary distribution.

\begin{lemma} \label{lmm:qsd}
Assume that conditions $\mathbf{A}$--$\mathbf{D}$ hold. Then:
\begin{itemize}
\item[$\mathbf{(i)}$] For sufficiently small $\varepsilon$, the quasi stationary distribution $\pi_j^{(\varepsilon)}$, given by relation \eqref{eq:qsddef}, have the following representation,
\begin{equation} \label{eq:qsdformula}
 \pi_j^{(\varepsilon)} = \frac{\omega_{i i j}^{(\varepsilon)}(\rho^{(\varepsilon)})}{\omega_{i i 1}^{(\varepsilon)}(\rho^{(\varepsilon)}) + \cdots + \omega_{i i N}^{(\varepsilon)}(\rho^{(\varepsilon)})}, \ i,j \neq 0.
\end{equation}
\item[$\mathbf{(ii)}$] For $j=1,\ldots,N$, we have
\begin{equation*}
 \pi_j^{(\varepsilon)} \rightarrow \pi_j^{(0)}, \ \text{as} \ \varepsilon \rightarrow 0.
\end{equation*}
\end{itemize}
\end{lemma}

\begin{proof}
Under condition $\mathbf{D}$, the functions $g_{i i}^{(0)}(n)$ are non-periodic for all $i \neq 0$. By Lemma \ref{lmm:convergence} we have that $\phi_{i i}^{(\varepsilon)}(\rho) \rightarrow \phi_{i i}^{(0)}(\rho)$ as $\varepsilon \rightarrow 0$, for $\rho \leq \delta$, $i \neq 0$. From this it follows that $g_{i i}^{(\varepsilon)}(n) \rightarrow g_{i i}^{(0)}(n)$ as $\varepsilon \rightarrow 0$, for $n \geq 0$, $i \neq 0$. Thus, we can conclude that there exists $\varepsilon_1 > 0$ such that the functions $\widetilde{g}_{i i}^{(\varepsilon)}(n)$, $i \neq 0$, are non-periodic for all $\varepsilon \leq \varepsilon_1$.

Now choose $\gamma$ such that $\rho^{(0)} < \gamma < \delta$. Using Lemmas \ref{lmm:chareq} and \ref{lmm:convergence}, we get the following for all $i \neq 0$,
\begin{equation*}
\begin{split}
 \limsup_{0 \leq \varepsilon \rightarrow 0} \sum_{n=0}^\infty n \widetilde{g}_{i i}^{(\varepsilon)}(n) &\leq \limsup_{0 \leq \varepsilon \rightarrow 0} \sum_{n=0}^\infty n e^{\gamma n} g_{i i}^{(\varepsilon)}(n) \\
 &\leq \left( \sup_{n \geq 0} n e^{-(\delta - \gamma)n} \right) \phi_{i i}^{(0)}(\delta) < \infty.
\end{split}
\end{equation*}
Thus, there exists $\varepsilon_2 > 0$ such that the distributions $\widetilde{g}_{i i}^{(\varepsilon)}(n)$, $i \neq 0$, have finite mean for all $\varepsilon \leq \varepsilon_2$.

Furthermore, it follows from Lemmas \ref{lmm:chareq} and \ref{lmm:convergence} that, for all $i,j \neq 0$,
\begin{equation*}
 \limsup_{0 \leq \varepsilon \rightarrow 0} \sum_{n=0}^\infty \widetilde{h}_{i j}^{(\varepsilon)}(n) \leq \limsup_{0 \leq \varepsilon \rightarrow 0} \sum_{n=0}^\infty e^{\delta n} h_{i j}^{(\varepsilon)}(n) = \omega_{i i j}^{(0)}(\delta) < \infty,
\end{equation*}
so there exists $\varepsilon_3 > 0$ such that $\sum_{n=0}^\infty \widetilde{h}_{i j}^{(\varepsilon)}(n) < \infty$, $i,j \neq 0$, for all $\varepsilon \leq \varepsilon_3$.

Now, let $\varepsilon_0 = \min \{ \varepsilon_1, \varepsilon_2, \varepsilon_3 \}$. For all $\varepsilon \leq \varepsilon_0$, the assumptions of the discrete time renewal theorem are satisfied for the renewal equation defined by \eqref{eq:properrenewaleq}. This yields
\begin{equation} \label{eq:renewaltheorem}
 \widetilde{P}_{i j}^{(\varepsilon)}(n) \rightarrow \frac{\sum_{k=0}^\infty \widetilde{h}_{i j}^{(\varepsilon)}(k)}{\sum_{k=0}^\infty k \widetilde{g}_{i i}^{(\varepsilon)}(k)}, \ \text{as} \ n \rightarrow \infty, \ i,j \neq 0, \ \varepsilon \leq \varepsilon_0.
\end{equation}

Note that we have
\begin{equation} \label{eq:condprob}
 \mathsf{P}_i \{ \xi^{(\varepsilon)}(n) = j \, | \, \mu_0^{(\varepsilon)} > n \} = \frac{\widetilde{P}_{i j}^{(\varepsilon)}(n)}{\sum_{k=1}^N \widetilde{P}_{i k}^{(\varepsilon)}(n)}, \ n = 0,1,\ldots, \ i,j \neq 0.
\end{equation}

It follows from \eqref{eq:renewaltheorem} and \eqref{eq:condprob} that, for $\varepsilon \leq \varepsilon_0$,
\begin{equation*}
 \mathsf{P}_i \{ \xi^{(\varepsilon)}(n) = j \, | \, \mu_0^{(\varepsilon)} > n \} \rightarrow \frac{\omega_{i i j}^{(\varepsilon)}(\rho^{(\varepsilon)})}{\sum_{k=1}^N \omega_{i i k}^{(\varepsilon)}(\rho^{(\varepsilon)})}, \ \text{as} \ n \rightarrow \infty, \ i,j \neq 0.
\end{equation*}

This proves part $\mathbf{(i)}$.

For the proof of part $\mathbf{(ii)}$, first note that,
\begin{equation} \label{eq:tailestimate}
\begin{split}
 0 &\leq \limsup_{0 \leq \varepsilon \rightarrow 0} \sum_{n=N}^\infty e^{\rho^{(\varepsilon)} n} h_{i j}^{(\varepsilon)}(n) \\
 &\leq \limsup_{0 \leq \varepsilon \rightarrow 0} \sum_{n=N}^\infty e^{\gamma n} h_{i j}^{(\varepsilon)}(n) \\
 &\leq e^{-(\delta - \gamma)N} \omega_{i i j}^{(0)}(\delta) < \infty, \ N=1,2,\ldots, \ i,j \neq 0.
\end{split}
\end{equation}

Relation \eqref{eq:tailestimate} implies that
\begin{equation} \label{eq:convtail}
 \lim_{N \rightarrow \infty} \limsup_{0 \leq \varepsilon \rightarrow 0} \sum_{n=N}^\infty e^{\rho^{(\varepsilon)} n} h_{i j}^{(\varepsilon)}(n) = 0, \ i,j \neq 0.
\end{equation}

It follows from Lemma \ref{lmm:chareq} that
\begin{equation} \label{eq:convrho}
 \rho^{(\varepsilon)} \rightarrow \rho^{(0)}, \ \text{as} \ \varepsilon \rightarrow 0.
\end{equation}

Since $h_{i j}^{(\varepsilon)}(n)$, for each $n=0,1,\ldots,$ can be written as a finite sum where each term in the sum is a continuous function of the quantities given in condition $\mathbf{A}$, we have
\begin{equation} \label{eq:convh}
 h_{i j}^{(\varepsilon)}(n) \rightarrow h_{i j}^{(0)}(n), \ \text{as} \ \varepsilon \rightarrow 0, \ i,j \neq 0.
\end{equation}

It now follows from \eqref{eq:convtail}, \eqref{eq:convrho}, and \eqref{eq:convh} that
\begin{equation} \label{eq:convomegarho}
 \omega_{i i j}^{(\varepsilon)}(\rho^{(\varepsilon)}) \rightarrow \omega_{i i j}^{(0)}(\rho^{(0)}), \ \text{as} \ \varepsilon \rightarrow 0, \ i,j \neq 0.
\end{equation}

Relations \eqref{eq:qsdformula} and \eqref{eq:convomegarho} show that part $\mathbf{(ii)}$ of Lemma \ref{lmm:qsd} holds.
\end{proof}

\section{Proof of the Main Result} \label{sec:proof}

In this section we prove Theorem \ref{thm:mainresult}.

Throughout this section, it is assumed that conditions $\mathbf{A}$--$\mathbf{D}$ and $\mathbf{P_{k+1}}$ hold.

The proof is given in a sequence of lemmas. For the proof of the first lemma, we refer to Petersson (2016).

\begin{lemma}
For $r=0,\ldots,k$ and $i,j \neq 0$, we have the following asymptotic expansions,
\begin{equation} \label{eq:expansionomegarho0}
 \omega_{i i j}^{(\varepsilon)}(\rho^{(0)},r) = a_{i j}[r,0] + a_{i j}[r,1] \varepsilon + \cdots + a_{i j}[r,k-r] \varepsilon^{k-r} + o(\varepsilon^{k-r}),
\end{equation}
\begin{equation} \label{eq:expansionphirho0}
 \phi_{i i}^{(\varepsilon)}(\rho^{(0)},r) = b_i[r,0] + b_i[r,1] \varepsilon + \cdots + b_i[r,k-r] \varepsilon^{k-r} + o(\varepsilon^{k-r}),
\end{equation}
where the coefficients in these expansions can be calculated from lemmas and theorems given in Petersson (2016).
\end{lemma}

Let us now recall from Section \ref{sec:quasistationary} that the quasi-stationary distribution, for sufficiently small $\varepsilon$, has the following representation,
\begin{equation} \label{eq:qsdformula2}
 \pi_j^{(\varepsilon)} = \frac{\omega_{i i j}^{(\varepsilon)}(\rho^{(\varepsilon)})}{\omega_{i i 1}^{(\varepsilon)}(\rho^{(\varepsilon)}) + \cdots + \omega_{i i N}^{(\varepsilon)}(\rho^{(\varepsilon)})}, \ j=1,\ldots,N.
\end{equation}

The construction of the asymptotic expansion for the quasi-stationary distribution will be realized in three steps. First, we use the coefficients in the expansions given by \eqref{eq:expansionphirho0} to build an asymptotic expansion for $\rho^{(\varepsilon)}$, the root of the characteristic equation. Then, the coefficients in this expansion and the coefficients in the expansions given by \eqref{eq:expansionomegarho0} are used to construct asymptotic expansions for $\omega_{i i j}^{(\varepsilon)}(\rho^{(\varepsilon)})$. Finally, relation \eqref{eq:qsdformula2} is used to complete the proof.

We formulate these steps in the following three lemmas. Let us here remark that the proof of Lemma \ref{lmm:expansionrho} is given in Silvestrov and Petersson (2013) in the context of general discrete time renewal equations and the proofs of Lemmas \ref{lmm:expansionomegarho} and \ref{lmm:expansionpi} are given in Petersson (2014b) in the context of quasi-stationary distributions for discrete time regenerative processes. In order to make the paper more self-contained, we also give the proofs here, in slightly reduced forms.

\begin{lemma} \label{lmm:expansionrho}
The root of the characteristic equation has the following asymptotic expansion,
\begin{equation*}
 \rho^{(\varepsilon)} = \rho^{(0)} + c_1 \varepsilon + \cdots + c_k \varepsilon^k + o(\varepsilon^k),
\end{equation*}
where $c_1 = -b_i[0,1] / b_i[1,0]$ and, for $n=2,\ldots,k,$
\begin{equation*}
\begin{split}
 c_n = - \frac{1}{b_i[1,0]} &\Bigg( b_i[0,n] + \sum_{q=1}^{n-1} b_i[1,n-q] c_q \\
 &+ \sum_{m=2}^n \sum_{q=m}^n b_i[m,n-q] \cdot \sum_{n_1,\ldots,n_{q-1} \in D_{m,q}} \prod_{p=1}^{q-1} \frac{c_p^{n_p}}{n_p!} \Bigg),
\end{split}
\end{equation*}
where $D_{m,q}$ is the set of all non-negative integer solutions of the system
\begin{equation*}
 n_1 + \cdots + n_{q-1} = m, \quad n_1 + 2 n_2 + \cdots + (q-1)n_{q-1} = q.
\end{equation*}
\end{lemma}

\begin{proof}
Let $\Delta^{(\varepsilon)} = \rho^{(\varepsilon)} - \rho^{(0)}$. It follows from the Taylor expansion of the exponential function that, for $n=0,1,\ldots,$
\begin{equation} \label{eq:lmmexpansionrho1}
 e^{\rho^{(\varepsilon)} n} = e^{\rho^{(0)} n} \left( \sum_{r=0}^k \frac{(\Delta^{(\varepsilon)})^r n^r}{r!} + \frac{(\Delta^{(\varepsilon)})^{k+1} n^{k+1}}{(k+1)!} e^{|\Delta^{(\varepsilon)}| n} \zeta_{k+1}^{(\varepsilon)}(n) \right),
\end{equation}
where $0 \leq \zeta_{k+1}^{(\varepsilon)}(n) \leq 1$.

If we multiply both sides of \eqref{eq:lmmexpansionrho1} by $g_{i i}^{(\varepsilon)}(n)$, sum over all $n$, and use that $\rho^{(\varepsilon)}$ is the root of the characteristic equation, we get
\begin{equation} \label{eq:lmmexpansionrho2}
 1 = \sum_{r=0}^k \frac{(\Delta^{(\varepsilon)})^r}{r!} \phi_{i i}^{(\varepsilon)}(\rho^{(0)},r) + (\Delta^{(\varepsilon)})^{k+1} M_{k+1}^{(\varepsilon)},
\end{equation}
where
\begin{equation} \label{eq:lmmexpansionrho3}
 M_{k+1}^{(\varepsilon)} = \frac{1}{(k+1)!} \sum_{n=0}^\infty n^{k+1} e^{(\rho^{(0)} + |\Delta^{(\varepsilon)}|) n} \zeta_{k+1}^{(\varepsilon)}(n) g_{i i}^{(\varepsilon)}(n).
\end{equation}

Let $\delta > \rho^{(0)}$ be the value from Lemma \ref{lmm:convergence}. It follows from Lemma \ref{lmm:chareq} that $|\Delta^{(\varepsilon)}| \rightarrow 0$ as $\varepsilon \rightarrow 0$, so there exist $\beta > 0$ and $\varepsilon_1(\beta) > 0$ such that
\begin{equation} \label{eq:lmmexpansionrho4}
 \rho^{(0)} + |\Delta^{(\varepsilon)}| \leq \beta < \delta, \ \varepsilon \leq \varepsilon_1(\beta).
\end{equation}

Since $\beta < \delta$, Lemma \ref{lmm:convergence} implies that there exists $\varepsilon_2(\beta) > 0$ such that
\begin{equation} \label{eq:lmmexpansionrho5}
 \phi_{i i}^{(\varepsilon)}(\beta,r) < \infty, \ r=0,1,\ldots, \ \varepsilon \leq \varepsilon_2(\beta).
\end{equation}

Let $\varepsilon_0 = \varepsilon_0(\beta) = \min \{ \varepsilon_1(\beta), \varepsilon_2(\beta) \}$. Then, relations \eqref{eq:lmmexpansionrho3}, \eqref{eq:lmmexpansionrho4}, and \eqref{eq:lmmexpansionrho5} imply that
\begin{equation} \label{eq:lmmexpansionrho6}
 M_{k+1}^{(\varepsilon)} \leq \frac{1}{(k+1)!} \phi_{i i}^{(\varepsilon)}(\beta,k+1) < \infty, \ \varepsilon \leq \varepsilon_0.
\end{equation}

It follows from \eqref{eq:lmmexpansionrho6} that we can rewrite \eqref{eq:lmmexpansionrho2} as
\begin{equation} \label{eq:lmmexpansionrho7}
 1 = \sum_{r=0}^k \frac{(\Delta^{(\varepsilon)})^r}{r!} \phi_{i i}^{(\varepsilon)}(\rho^{(0)},r) + (\Delta^{(\varepsilon)})^{k+1} M_{k+1} \zeta_{k+1}^{(\varepsilon)},
\end{equation}
where $M_{k+1} = \sup_{\varepsilon \leq \varepsilon_0} M_{k+1}^{(\varepsilon)} < \infty$ and $0 \leq \zeta_{k+1}^{(\varepsilon)} \leq 1$.

From relation \eqref{eq:lmmexpansionrho7} we can successively construct the asymptotic expansion for the root of the characteristic equation.

Let us first assume that $k=1$. In this case \eqref{eq:lmmexpansionrho7} implies that
\begin{equation} \label{eq:lmmexpansionrho8}
 1 = \phi_{i i}^{(\varepsilon)}(\rho^{(0)},0) + \Delta^{(\varepsilon)} \phi_{i i}^{(\varepsilon)}(\rho^{(0)},1) + (\Delta^{(\varepsilon)})^2 O(1).
\end{equation}

Using \eqref{eq:expansionphirho0}, \eqref{eq:lmmexpansionrho8}, and that $\Delta^{(\varepsilon)} \rightarrow 0$ as $\varepsilon \rightarrow 0$, it follows that
\begin{equation} \label{eq:lmmexpansionrho9}
 - b_i[0,1] \varepsilon = \Delta^{(\varepsilon)} (b_i[1,0] + o(1)) + o(\varepsilon).
\end{equation}

Dividing both sides of equation \eqref{eq:lmmexpansionrho9} by $\varepsilon$ and letting $\varepsilon$ tend to zero, we can conclude that $\Delta^{(\varepsilon)} / \varepsilon \rightarrow -b_i[0,1] / b_i[1,0]$ as $\varepsilon \rightarrow 0$. From this it follows that we have the representation
\begin{equation} \label{eq:lmmexpansionrho10}
 \Delta^{(\varepsilon)} = c_1 \varepsilon + \Delta_1^{(\varepsilon)},
\end{equation}
where $c_1 = -b_i[0,1] / b_i[1,0]$ and $\Delta_1^{(\varepsilon)} / \varepsilon \rightarrow 0$ as $\varepsilon \rightarrow 0$.

This proves Lemma \ref{lmm:expansionrho} for the case $k=1$.

Let us now assume that $k=2$. In this case relation \eqref{eq:lmmexpansionrho7} implies that
\begin{equation} \label{eq:lmmexpansionrho11}
 1 = \phi_{i i}^{(\varepsilon)}(\rho^{(0)},0) + \Delta^{(\varepsilon)} \phi_{i i}^{(\varepsilon)}(\rho^{(0)},1) + \frac{(\Delta^{(\varepsilon)})^2}{2} \phi_{i i}^{(\varepsilon)}(\rho^{(0)},2) + (\Delta^{(\varepsilon)})^3 O(1).
\end{equation}

Using \eqref{eq:expansionphirho0} and \eqref{eq:lmmexpansionrho10} in relation \eqref{eq:lmmexpansionrho11} and rearranging gives
\begin{equation} \label{eq:lmmexpansionrho12}
 - \left( b_i[0,2] + b_i[1,1] c_1 + \frac{b_i[2,0] c_1^2}{2} \right) \varepsilon^2 = \Delta_1^{(\varepsilon)} ( b_i[1,0] + o(1) ) + o(\varepsilon^2).
\end{equation}

Dividing both sides of equation \eqref{eq:lmmexpansionrho12} by $\varepsilon^2$ and letting $\varepsilon$ tend to zero, we can conclude that $\Delta_1^{(\varepsilon)} / \varepsilon^2 \rightarrow c_2$ as $\varepsilon \rightarrow 0$, where
\begin{equation*}
 c_2 = - \frac{1}{b_i[1,0]} \left( b_i[0,2] + b_i[1,1] c_1 + \frac{b_i[2,0] c_1^2}{2} \right).
\end{equation*}

From this and \eqref{eq:lmmexpansionrho10} it follows that we have the representation
\begin{equation*}
 \Delta^{(\varepsilon)} = c_1 \varepsilon + c_2 \varepsilon^2 + \Delta_2^{(\varepsilon)},
\end{equation*}
where $\Delta_2^{(\varepsilon)} / \varepsilon^2 \rightarrow 0$ as $\varepsilon \rightarrow 0$.

This proves Lemma \ref{lmm:expansionrho} for the case $k=2$.

Continuing in this way we can prove the lemma for any positive integer $k$. However, once it is known that the expansion exists, the coefficients can be obtained in a simpler way. From \eqref{eq:expansionphirho0} and \eqref{eq:lmmexpansionrho7} we get the following formal equation,
\begin{equation} \label{eq:lmmexpansionrho13}
\begin{split}
 &- (b_i[0,1] \varepsilon + b_i[0,2] \varepsilon^2 + \cdots ) \\
 &\quad \quad = ( c_1 \varepsilon + c_2 \varepsilon^2 + \cdots ) ( b_i[1,0] + b_i[1,1] \varepsilon + \cdots ) \\
 &\quad \quad + (1/2!)( c_1 \varepsilon + c_2 \varepsilon^2 + \cdots )^2 ( b_i[2,0] + b_i[2,1] \varepsilon + \cdots ) + \cdots
\end{split}
\end{equation}
By expanding the right hand side of \eqref{eq:lmmexpansionrho13} and then equating coefficients of equal powers of $\varepsilon$ in the left and right hand sides, we obtain the formulas given in Lemma \ref{lmm:expansionrho}.
\end{proof}

\begin{lemma} \label{lmm:expansionomegarho}
For any $i,j \neq 0$, we have the following asymptotic expansion,
\begin{equation*}
 \omega_{i i j}^{(\varepsilon)}(\rho^{(\varepsilon)}) = \omega_{i i j}^{(0)}(\rho^{(0)}) + d_{i j}[1] \varepsilon + \cdots + d_{i j}[k] \varepsilon^k + o(\varepsilon^k),
\end{equation*}
where $d_{i j}[1] = a_{i j}[0,1] + a_{i j}[1,0] c_1$, and, for $n=2,\ldots,k$,
\begin{equation*}
\begin{split}
 d_{i j}[n] &= a_{i j}[0,n] + \sum_{q=1}^n a_{i j}[1,n-q] c_q \\
 &+ \sum_{m=2}^n \sum_{q=m}^n a_{i j}[m,n-q] \cdot \sum_{n_1,\ldots,n_{q-1} \in D_{m,q}} \prod_{p=1}^{q-1} \frac{c_p^{n_p}}{n_p!},
\end{split}
\end{equation*}
where $D_{m,q}$ is the set of all non-negative integer solutions of the system
\begin{equation*}
 n_1 + \cdots + n_{q-1} = m, \quad n_1 + 2 n_2 + \cdots + (q-1)n_{q-1} = q.
\end{equation*}
\end{lemma}

\begin{proof}
Let us again use relation \eqref{eq:lmmexpansionrho1} given in the proof of Lemma \ref{lmm:expansionrho}. Multiplying both sides of \eqref{eq:lmmexpansionrho1} by $h_{i j}^{(\varepsilon)}(n)$ and summing over all $n$ we get
\begin{equation} \label{eq:lmmexpansionomegarho1}
 \omega_{i i j}^{(\varepsilon)}(\rho^{(\varepsilon)}) = \sum_{r=0}^k \frac{(\Delta^{(\varepsilon)})^r}{r!} \omega_{i i j}^{(\varepsilon)}(\rho^{(0)},r) + (\Delta^{(\varepsilon)})^{k+1} \widetilde{M}_{k+1}^{(\varepsilon)},
\end{equation}
where
\begin{equation*}
 \widetilde{M}_{k+1}^{(\varepsilon)} = \frac{1}{(k+1)!} \sum_{n=0}^\infty n^{k+1} e^{(\rho^{(0)} + |\Delta^{(\varepsilon)}|) n} \zeta_{k+1}^{(\varepsilon)}(n) h_{i j}^{(\varepsilon)}(n).
\end{equation*}

Using similar arguments as in the proof of Lemma \ref{lmm:expansionrho} we can rewrite \eqref{eq:lmmexpansionomegarho1} as
\begin{equation} \label{eq:lmmexpansionomegarho2}
 \omega_{i i j}^{(\varepsilon)}(\rho^{(\varepsilon)}) = \sum_{r=0}^k \frac{(\Delta^{(\varepsilon)})^r}{r!} \omega_{i i j}^{(\varepsilon)}(\rho^{(0)},r) + (\Delta^{(\varepsilon)})^{k+1} \widetilde{M}_{k+1} \zeta_{k+1}^{(\varepsilon)},
\end{equation}
where $\widetilde{M}_{k+1} = \sup_{\varepsilon \leq \varepsilon_0} \widetilde{M}_{k+1}^{(\varepsilon)} < \infty$, for some $\varepsilon_0 > 0$, and $0 \leq \zeta_{k+1}^{(\varepsilon)} \leq 1$.

From Lemma \ref{lmm:expansionrho} we have the following asymptotic expansion,
\begin{equation} \label{eq:lmmexpansionomegarho3}
 \Delta^{(\varepsilon)} = c_1 \varepsilon + \cdots + c_k \varepsilon^k + o(\varepsilon^k).
\end{equation}

Substituting the expansions \eqref{eq:expansionomegarho0} and \eqref{eq:lmmexpansionomegarho3} into relation \eqref{eq:lmmexpansionomegarho2} yields
\begin{equation} \label{eq:lmmexpansionomegarho4}
\begin{split}
 \omega_{i i j}^{(\varepsilon)}(\rho^{(\varepsilon)}) &= \omega_{i i j}^{(0)}(\rho^{(0)}) + a_{i j}[0,1] \varepsilon + \cdots + a_{i j}[0,k] \varepsilon^k + o(\varepsilon^k) \\
 &+ ( c_1 \varepsilon + \cdots + c_k \varepsilon^k + o(\varepsilon^k) ) \\
 &\times ( a_{i j}[1,0] + a_{i j}[1,1] \varepsilon + \cdots + a_{i j}[1,k-1] \varepsilon^{k-1} + o(\varepsilon^{k-1}) ) \\
 &+ \cdots + \\
 &+ (1/k!)( c_1 \varepsilon + \cdots + c_k \varepsilon^k + o(\varepsilon^k) )^k ( a_{i j}[k,0] + o(1) ).
\end{split}
\end{equation}
By expanding the right hand side of \eqref{eq:lmmexpansionomegarho4} and grouping coefficients of equal powers of $\varepsilon$ we get the expansions and formulas given in Lemma \ref{lmm:expansionomegarho}.
\end{proof}

\begin{lemma} \label{lmm:expansionpi}
For any $j \neq 0$, we have the following asymptotic expansion,
\begin{equation} \label{eq:expansionpi}
 \pi_j^{(\varepsilon)} = \pi_j^{(0)} + \pi_j[1] \varepsilon + \cdots + \pi_j[k] \varepsilon^k + o(\varepsilon^k).
\end{equation}
The coefficients $\pi_j[n]$, $n=1,\ldots,k$, $j \neq 0$, are for any $i \neq 0$ given by the following recursive formulas,
\begin{equation*}
 \pi_j[n] = \frac{1}{e_i[0]} \left( d_{i j}[n] - \sum_{q=0}^{n-1} e_i[n-q] \pi_j[q] \right), \ n=1,\ldots,k,
\end{equation*}
where $\pi_j[0] = \pi_j^{(0)}$, $d_{i j}[0] = \omega_{i i j}^{(0)}(\rho^{(0)})$, and $e_i[n] = \sum_{j \neq 0} d_{i j}[n]$, $n=0,\ldots,k$.
\end{lemma}

\begin{proof}
It follows from formula \eqref{eq:qsdformula2} and Lemma \ref{lmm:expansionomegarho} that we for all $i,j \neq 0$ have
\begin{equation} \label{eq:lmmexpansionpi1}
 \pi_j^{(\varepsilon)} = \frac{d_{i j}[0] + d_{i j}[1] \varepsilon + \cdots + d_{i j}[k] \varepsilon^k + o(\varepsilon^k)}{e_i[0] + e_i[1] \varepsilon + \cdots + e_i[k] \varepsilon^k + o(\varepsilon^k)}.
\end{equation}

Since $e_i[0] > 0$, it follows from \eqref{eq:lmmexpansionpi1} that the expansion \eqref{eq:expansionpi} exists. From this and \eqref{eq:lmmexpansionpi1} we get the following equation,
\begin{equation} \label{eq:lmmexpansionpi2}
\begin{split}
 &( e_i[0] + e_i[1] \varepsilon + \cdots + e_i[k] \varepsilon^k + o(\varepsilon^k) ) \\
 &\quad \quad \times ( \pi_j[0] + \pi_j[1] \varepsilon + \cdots + \pi_j[k] \varepsilon^k + o(\varepsilon^k) ) \\
 &\quad \quad = d_{i j}[0] + d_{i j}[1] \varepsilon + \cdots + d_{i j}[k] \varepsilon^k + o(\varepsilon^k).
\end{split}
\end{equation}
By expanding the left hand side of \eqref{eq:lmmexpansionpi2} and then equating coefficients of equal powers of $\varepsilon$ in the left and right hand sides, we obtain the coefficients given in Lemma \ref{lmm:expansionpi}.
\end{proof}

\section{Perturbed Markov Chains} \label{sec:markovchains}

In this section it is shown how the results of the present paper can be applied in the special case of perturbed discrete time Markov chains. As an illustration, we present a simple numerical example.

For every $\varepsilon \geq 0$, let $\eta_n^{(\varepsilon)}$, $n=0,1,\ldots,$ be a homogeneous discrete time Markov chain with state space $X = \{0,1,\ldots,N\}$, an initial distribution $p_i^{(\varepsilon)} = \mathsf{P} \{ \eta_0^{(\varepsilon)} = i \}$, $i \in X$, and transition probabilities
\begin{equation*}
 p_{i j}^{(\varepsilon)} = \mathsf{P} \{ \eta_{n+1}^{(\varepsilon)} = j \, | \, \eta_n^{(\varepsilon)} = i \}, \ i,j \in X.
\end{equation*}

This model is a particular case of a semi-Markov process. In this case, the transition probabilities are given by
\begin{equation*}
 Q_{i j}^{(\varepsilon)}(n) = p_{i j}^{(\varepsilon)} \chi( n=1 ), \ n=1,2,\ldots, \ i,j \in X.
\end{equation*}

Furthermore, mixed power-exponential moment functionals for transition probabilities take the following form,
\begin{equation} \label{eq:defpmc}
 p_{i j}^{(\varepsilon)}(\rho,r) = \sum_{n=0}^\infty n^r e^{\rho n} Q_{i j}^{(\varepsilon)}(n) = e^{\rho} p_{i j}^{(\varepsilon)}, \ \rho \in \mathbb{R}, \ r=0,1,\ldots, \ i,j \in X.
\end{equation}

Conditions $\mathbf{A}$--$\mathbf{D}$ and $\mathbf{P_k}$ imposed in Section \ref{sec:mainresult} now hold if the following conditions are satisfied:

\begin{enumerate}
\item[$\mathbf{A'}$:] $g_{i j}^{(0)} > 0$, $i,j \neq 0$. 
\end{enumerate}

\begin{enumerate}
\item[$\mathbf{B'}$:] $g_{i i}^{(0)}(n)$ is non-periodic for some $i \neq 0$.
\end{enumerate}

\begin{enumerate}
\item[$\mathbf{P_k'}$:] $p_{i j}^{(\varepsilon)} = p_{i j}^{(0)} + p_{i j}[1] \varepsilon + \cdots + p_{i j}[k] \varepsilon^k + o(\varepsilon^k)$, $i,j \neq 0$, where $|p_{i j}[n]| < \infty$, $n=1,\ldots,k$, $i,j \neq 0$.
\end{enumerate}

Let us here remark that in order to construct an asymptotic expansion of order $k$ for the quasi-stationary distribution of a Markov chain, it is sufficient to assume that the perturbation condition holds for the parameter $k$, and not for $k+1$ as needed for semi-Markov processes. The stronger perturbation condition with parameter $k+1$ is needed in order to construct the asymptotic expansions given in equation \eqref{eq:expansionomegarho0}. However, for Markov chains these expansions can be constructed under the weaker perturbation condition. This follows from results given in Petersson (2016).

It follows from \eqref{eq:defpmc} and $\mathbf{P_k'}$ that the coefficients in the perturbation condition $\mathbf{P_k}$ are given by
\begin{equation} \label{eq:coeffrelation}
 p_{i j}[\rho^{(0)},r,n] = e^{\rho^{(0)}} p_{i j}[n], \ r=0,\ldots,k, \ n=0,\ldots,k-r, \ i,j \neq 0.
\end{equation}

Let us illustrate the remarks made above by means of a simple numerical example where we compute the asymptotic expansion of second order for the quasi-stationary distribution of a Markov chain with four states. We consider the simplest case where transitions to state $0$ is not possible for the limiting Markov chain. In this case, exact computations can be made and we can focus on the algorithm itself and need not need to consider possible numerical issues.

We consider a perturbed Markov chain $\eta_n^{(\varepsilon)}$, $n=0,1,\ldots,$ on the state space $X = \{ 0,1,2,3 \}$ with a matrix of transition probabilities given by
\begin{equation} \label{eq:transitionprobmc}
 \| p_{i j}^{(\varepsilon)} \| =
 \begin{bmatrix}
  1 & 0 & 0 & 0 \\
  1-e^{-\varepsilon} & 0 & e^{-\varepsilon} & 0 \\
  1-e^{-\varepsilon} & 0 & 0 & e^{-\varepsilon} \\
  1-e^{-2 \varepsilon} & \frac{1}{2} e^{-2 \varepsilon} & \frac{1}{2} e^{-2 \varepsilon} & 0
 \end{bmatrix}
 , \ \varepsilon \geq 0.
\end{equation}

First, using the well known asymptotic expansion for the exponential function, we obtain the coefficients in condition $\mathbf{P_k'}$. The non-zero coefficients in this condition take the following numerical values,
\begin{equation} \label{eq:coeffp}
\begin{array}{l l l l}
 p_{1 2}[0] = 1, & p_{2 3}[0] = 1, & p_{3 1}[0] = 1/2, & p_{3 2}[0] = 1/2, \\
 p_{1 2}[1] = -1, & p_{2 3}[1] = -1, & p_{3 1}[1] = -1, & p_{3 2}[1] = -1, \\
 p_{1 2}[2] = 1/2, & p_{2 3}[2] = 1/2, & p_{3 1}[2] = 1, & p_{3 2}[2] = 1.
\end{array}
\end{equation}

Then, the root of the characteristic equation for the limiting Markov chain needs to be found. Since $\phi_{i i}^{(0)}(0)  = \mathsf{P}_i \{ \nu_0^{(0)} > \nu_i^{(0)} \} = 1$, we have $\rho^{(0)} = 0$. In the case where transitions to state $0$ is possible also for the limiting Markov chain, the root $\rho^{(0)}$ needs to be computed numerically.

Now, using that $\rho^{(0)} = 0$ and relations \eqref{eq:coeffrelation} and \eqref{eq:coeffp}, we obtain the coefficients in condition $\mathbf{P_k}$.

Next step is to determine the coefficients in the expansions given in equations \eqref{eq:expansionomegarho0} and \eqref{eq:expansionphirho0} for the case where $k=2$ and $i$ is some fixed state which we can choose arbitrarily. Let us choose $i=1$. In order to compute these coefficients we apply the results given in Petersson (2016). According to these results, we can, based on the coefficients in condition $\mathbf{P_k}$, compute the following asymptotic vector expansions,
\begin{equation} \label{eq:expansionphimc}
\begin{split}
 &\mathsf{\Phi}_1^{(\varepsilon)}(0,0) = \mathsf{\Phi}_1[0,0,0] + \mathsf{\Phi}_1[0,0,1] \varepsilon + \mathsf{\Phi}_1[0,0,2] \varepsilon^2 + \mathbf{o}(\varepsilon^2), \\
 &\mathsf{\Phi}_1^{(\varepsilon)}(0,1) = \mathsf{\Phi}_1[0,1,0] + \mathsf{\Phi}_1[0,1,1] \varepsilon + \mathbf{o}(\varepsilon), \\
 &\mathsf{\Phi}_1^{(\varepsilon)}(0,2) = \mathsf{\Phi}_1[0,2,0] + \mathbf{o}(1),
\end{split}
\end{equation}
and, for $j=1,2,3$,
\begin{equation} \label{eq:expansionomegamc}
\begin{split}
 &\boldsymbol{\omega}_{1 j}^{(\varepsilon)}(0,0) = \boldsymbol{\omega}_{1 j}[0,0,0] + \boldsymbol{\omega}_{1 j}[0,0,1] \varepsilon + \boldsymbol{\omega}_{1 j}[0,0,2] \varepsilon^2 + \mathbf{o}(\varepsilon^2), \\
 &\boldsymbol{\omega}_{1 j}^{(\varepsilon)}(0,1) = \boldsymbol{\omega}_{1 j}[0,1,0] + \boldsymbol{\omega}_{1 j}[0,1,1] \varepsilon + \mathbf{o}(\varepsilon), \\
 &\boldsymbol{\omega}_{1 j}^{(\varepsilon)}(0,2) = \boldsymbol{\omega}_{1 j}[0,2,0] + \mathbf{o}(1),
\end{split}
\end{equation}
where
\begin{equation*}
 \mathsf{\Phi}_1^{(\varepsilon)}(0,r) =
 \begin{bmatrix}
  \phi_{1 1}^{(\varepsilon)}(0,r) & \phi_{2 1}^{(\varepsilon)}(0,r) & \phi_{3 1}^{(\varepsilon)}(0,r)
 \end{bmatrix}
 ^T, \ r=0,1,2,
\end{equation*}
and
\begin{equation*}
 \boldsymbol{\omega}_{1 j}^{(\varepsilon)}(0,r) =
 \begin{bmatrix}
  \omega_{1 1 j}^{(\varepsilon)}(0,r) & \omega_{2 1 j}^{(\varepsilon)}(0,r) & \omega_{3 1 j}^{(\varepsilon)}(0,r)
 \end{bmatrix}
 ^T, \ r=0,1,2, \ j=1,2,3.
\end{equation*}

For example, the coefficients in \eqref{eq:expansionphimc} take the following numerical values,
\begin{equation} \label{eq:coeffphi}
\begin{array}{l l l}
 \mathsf{\Phi}_1[0,0,0] =
 \begin{bmatrix}
  1 \\ 1 \\ 1
 \end{bmatrix}
 , &
 \mathsf{\Phi}_1[0,0,1] =
 \begin{bmatrix}
  -7 \\ -6 \\ -5
 \end{bmatrix}
 , &
 \mathsf{\Phi}_1[0,0,2] =
 \begin{bmatrix}
  67/2 \\ 27 \\ 43/2
 \end{bmatrix}
 , \\ \\
 \mathsf{\Phi}_1[0,1,0] =
 \begin{bmatrix}
  5 \\ 4 \\ 3
 \end{bmatrix}
 , &
 \mathsf{\Phi}_1[0,1,1] =
 \begin{bmatrix}
  -47 \\ -36 \\ -27
 \end{bmatrix}
 , &
 \mathsf{\Phi}_1[0,2,0] =
 \begin{bmatrix}
  33 \\ 24 \\ 17
 \end{bmatrix}
 .
\end{array}
\end{equation}

In particular, from \eqref{eq:expansionphimc} and \eqref{eq:expansionomegamc} we can extract the following asymptotic expansions,
\begin{equation*}
\begin{split}
 &\phi_{1 1}^{(\varepsilon)}(0,0) = b_1[0,0] + b_1[0,1] \varepsilon + b_1[0,2] \varepsilon^2 + o(\varepsilon^2), \\
 &\phi_{1 1}^{(\varepsilon)}(0,1) = b_1[1,0] + b_1[1,1] \varepsilon + o(\varepsilon), \\
 &\phi_{1 1}^{(\varepsilon)}(0,2) = b_1[2,0] + o(1),
\end{split}
\end{equation*}
and, for $j=1,2,3$,
\begin{equation*}
\begin{split}
 &\omega_{1 1 j}^{(\varepsilon)}(0,0) = a_{1 j}[0,0] + a_{1 j}[0,1] \varepsilon + a_{1 j}[0,2] \varepsilon^2 + o(\varepsilon^2), \\
 &\omega_{1 1 j}^{(\varepsilon)}(0,1) = a_{1 j}[1,0] + a_{1 j}[1,1] \varepsilon + o(\varepsilon), \\
 &\omega_{1 1 j}^{(\varepsilon)}(0,2) = a_{1 j}[2,0] + o(1).
\end{split}
\end{equation*}

From \eqref{eq:expansionphimc} and \eqref{eq:coeffphi} it follows that
\begin{equation} \label{eq:coeffb}
\begin{array}{l l l}
 b_1[0,0] = 1, & b_1[0,1] = -7, & b_1[0,2] = 67/2, \\
 b_1[1,0] = 5, & b_1[1,1] = -47, & b_1[2,0] = 33.
\end{array}
\end{equation}

By first calculating the coefficients in \eqref{eq:expansionomegamc}, we then get the following numerical values,
\begin{equation} \label{eq:coeffa}
\begin{array}{l l l}
 a_{1 1}[0,0] = 1, & a_{1 2}[0,0] = 2, & a_{1 3}[0,0] = 2, \\
 a_{1 1}[0,1] = 0, & a_{1 2}[0,1] = -8, & a_{1 3}[0,1] = -10, \\
 a_{1 1}[0,2] = 0, & a_{1 2}[0,2] = 34, & a_{1 3}[0,2] = 43, \\
 a_{1 1}[1,0] = 0, & a_{1 2}[1,0] = 6, & a_{1 3}[1,0] = 8, \\
 a_{1 1}[1,1] = 0, & a_{1 2}[1,1] = -48, & a_{1 3}[1,1] = -64, \\
 a_{1 1}[2,0] = 0, & a_{1 2}[2,0] = 34, & a_{1 3}[2,0] = 48.
\end{array}
\end{equation}

The asymptotic expansion for the quasi-stationary distribution can now be computed from the coefficients in equations \eqref{eq:coeffb} and \eqref{eq:coeffa} by applying the lemmas in Section \ref{sec:proof}.

From Lemma \ref{lmm:expansionrho} we get that the asymptotic expansion for the root of the characteristic equation is given by
\begin{equation*}
 \rho^{(\varepsilon)} = c_1 \varepsilon + c_2 \varepsilon^2 + o(\varepsilon^2),
\end{equation*}
where
\begin{equation} \label{eq:coeffc}
 c_1 = - \frac{b_1[0,1]}{b_1[1,0]} = \frac{7}{5}, \, c_2 = - \frac{b_1[0,2] + b_1[1,1] c_1 + b_1[2,0] c_1^2 / 2}{b_1[1,0]} = - \frac{1}{125}.
\end{equation}

Then, Lemma \ref{lmm:expansionomegarho} gives us the following asymptotic expansions,
\begin{equation*}
 \omega_{1 1 j}^{(\varepsilon)}(\rho^{(\varepsilon)}) = d_{1 j}[0] + d_{1 j}[1] \varepsilon + d_{1 j}[2] \varepsilon^2 + o(\varepsilon^2), \ j=1,2,3,
\end{equation*}
where
\begin{equation} \label{eq:coeffdformulas}
\begin{split}
 &d_{1 j}[0] = a_{1 j}[0,0], \\
 &d_{1 j}[1] = a_{1 j}[0,1] + a_{1 j}[1,0] c_1, \\
 &d_{1 j}[2] = a_{1 j}[0,2] + a_{1 j}[1,1] c_1 + a_{1 j}[1,0] c_2 + a_{1 j}[2,0] c_1^2 / 2.
\end{split}
\end{equation}

From \eqref{eq:coeffa}, \eqref{eq:coeffc}, and \eqref{eq:coeffdformulas}, we calculate
\begin{equation} \label{eq:coeffd}
\begin{array}{l l l}
 d_{1 1}[0] = 1, & d_{1 2}[0] = 2, & d_{1 3}[0] = 2, \\
 d_{1 1}[1] = 0, & d_{1 2}[1] = 2/5, & d_{1 3}[1] = 6/5, \\
 d_{1 1}[2] = 0, & d_{1 2}[2] = 9/125, & d_{1 3}[2] = 47/125.
\end{array}
\end{equation}

Finally, let us use Lemma \ref{lmm:expansionpi}. First, using \eqref{eq:coeffd}, we get
\begin{equation} \label{eq:coeffe}
\begin{split}
 &e_1[0] = d_{1 1}[0] + d_{1 2}[0] + d_{1 3}[0] = 5, \\
 &e_1[1] = d_{1 1}[1] + d_{1 2}[1] + d_{1 3}[1] = 8/5, \\
 &e_1[2] = d_{1 1}[2] + d_{1 2}[2] + d_{1 3}[2] = 56/125.
\end{split}
\end{equation}

Then, we can construct the asymptotic expansion for the quasi-stationary distribution,
\begin{equation*}
 \pi_j^{(\varepsilon)} = \pi_j[0] + \pi_j[1] \varepsilon + \pi_j[2] \varepsilon^2 + o(\varepsilon^2), \ j=1,2,3,
\end{equation*}
where
\begin{equation} \label{eq:coeffpiformulas}
\begin{split}
 &\pi_j[0] = d_{1 j}[0] / e_1[0], \\
 &\pi_j[1] = ( d_{1 j}[1] - e_1[1] \pi_j[0] ) / e_1[0], \\
 &\pi_j[2] = ( d_{1 j}[2] - e_1[2] \pi_j[0] - e_1[1] \pi_j[1] ) / e_1[0].
\end{split}
\end{equation}

Using \eqref{eq:coeffd}, \eqref{eq:coeffe}, and \eqref{eq:coeffpiformulas}, the following numerical values are obtained,
\begin{equation*}
\begin{array}{l l l}
 \pi_{1}[0] = 1/5, & \pi_{2}[0] = 2/5, & \pi_{3}[0] = 2/5, \\
 \pi_{1}[1] = -8/125, & \pi_{2}[1] = -6/125, & \pi_{3}[1] = 14/125, \\
 \pi_{1}[2] = 8/3125, & \pi_{2}[2] = -19/3125, & \pi_{3}[2] = 11/3125.
\end{array}
\end{equation*}
Note here that $(\pi_1[0],\pi_2[0],\pi_3[0])$ is the stationary distribution of the limiting Markov chain. It is also worth noticing that $\pi_1[n] + \pi_2[n] + \pi_3[n] = 0$ for $n=1,2$, as expected.

% Bibliography

\end{document}